\documentclass[a4paper,10pt,reqno]{amsart}

\usepackage{amsthm,amsmath,amsfonts,amssymb}
\usepackage{color}

\usepackage{todonotes}

\usepackage[colorlinks=true,urlcolor=blue,
citecolor=red,linkcolor=blue,linktocpage,pdfpagelabels,
bookmarksnumbered,bookmarksopen]{hyperref}%\usepackage[final]{showlabels}%cambiare argomento a final per disabilitarlo %%Il pacchetto va sempre inserito DOPO hyperref e amsmath
\usepackage{pgf, tikz-cd}
\usepackage{tikzscale}

\usepackage{mathrsfs}
\usetikzlibrary{arrows}

\usepackage{subfigure}

\definecolor{grey}{rgb}{.7,.7,.7}

\numberwithin{equation}{section}
\newcommand{\aplim}{\mathop{\text{\rm ap-lim}}}

\newcommand{\e}{\varepsilon}
\newcommand{\N}{\mathbb{N}}
\newcommand{\de}{\partial}

\renewcommand{\phi}{\varphi}

\DeclareMathOperator{\spt}{spt}

\newcommand{\Om}{\Omega}
\let\div\relax \DeclareMathOperator{\div}{div}
\renewcommand{\epsilon}{\varepsilon}
\newcommand{\DM}{\mathcal{DM}^\infty}

\newcommand{\R}{\mathbb{R}}
\newcommand{\Rnp}{{\R^n_+}}
\newcommand{\Rno}{{\R^n_0}}
\newcommand{\Rnm}{{\R^n_-}}
\newcommand{\Leb}{{\mathcal L}}

\def\H{\mathcal{H}}

\theoremstyle{plain}
\newtheorem{thm}{Theorem}[section]
\newtheorem{lem}[thm]{Lemma}
\newtheorem{prop}[thm]{Proposition}
\newtheorem{cor}[thm]{Corollary}
\newtheorem*{thm*}{Theorem}

\theoremstyle{definition}
\newtheorem{defin}[thm]{Definition}

\theoremstyle{remark}
\newtheorem{rem}[thm]{Remark}
\newtheorem{ex}[thm]{Example}
\newtheorem*{ex*}{Example}

\title[Rigidity and trace properties of divergence-measure vector fields]{
Rigidity and trace properties \\ of divergence-measure vector fields}

\author{Gian Paolo Leonardi}
\address{Dipartimento di Matematica, Universit\`a di Trento, via Sommarive 14, I-38123 Povo (TN), ITALY}
\email{gianpaolo.leonardi@unitn.it}

\author{Giorgio Saracco}
\address{Dipartimento di Matematica, Universit\`a di Pavia, via Ferrata 5, I-27100 Pavia (PV), ITALY}
\email{giorgio.saracco@unipv.it}

\thanks{G.P. Leonardi and G. Saracco have been partially supported by the INdAM-GNAMPA 2019 project ``Problemi isoperimetrici in spazi Euclidei e non'' (n.~prot.~U-UFMBAZ-2019-000473 11-03-2019).}


\begin{document}

\definecolor{eqeqeq}{rgb}{0.87,0.87,0.87}
\definecolor{ffffff}{rgb}{1.,1.,1.}
\definecolor{black}{rgb}{0.,0.,0.}
\definecolor{zzffzz}{rgb}{0.6,1.,0.6}
\definecolor{ttqqqq}{rgb}{0.2,0.,0.}
\definecolor{qqqqtt}{rgb}{0.,0.,0.2}
\definecolor{uuuuuu}{rgb}{0.26,0.26,0.26}

\begin{abstract}
We consider a $\phi$-rigidity property for divergence-free vector fields in the Euclidean $n$-space, where $\phi(t)$ is a non-negative convex function vanishing only at $t=0$. We show that this property is always satisfied in dimension $n=2$, while in higher dimension it requires some further restriction on $\phi$. 
In particular, we exhibit counterexamples to \textit{quadratic rigidity} (i.e.~when $\phi(t) = ct^2$) in dimension $n\ge 4$. The validity of the quadratic rigidity, which we prove in dimension $n=2$, implies the existence of the trace of a divergence-measure vector field $\xi$ on an $\H^{1}$-rectifiable set $S$, as soon as its weak normal trace $[\xi\cdot \nu_S]$ is maximal on $S$. As an application, we deduce that the graph of an extremal solution to the prescribed mean curvature equation in a weakly-regular domain becomes vertical near the boundary in a pointwise sense.
\end{abstract}

 \hspace{-3cm}
 {
 \begin{minipage}[t]{0.6\linewidth}
 \begin{scriptsize}
 \vspace{-2cm}
 This is a pre-print of an article published in \emph{Adv. Calc. Var.}. The final authenticated version is available online at: \href{https://doi.org/10.1515/acv-2019-0094}{https://doi.org/10.1515/acv-2019-0094}
 \end{scriptsize}
\end{minipage} 
}

\maketitle

\section{Introduction}
The structure and the properties of vector fields, whose distributional divergence either vanishes or is represented by a locally finite measure, are of great interest in Mathematics and in Physics. Such vector fields arise, for instance, in fluid mechanics, in electromagnetism, and in conservation laws. We shall not give a detailed account, however the interested reader is referred to the monographs~\cite{feireisl2009singular, Galdi2011book}, as well as to the papers~\cite{ChaeWolf2016, DeLellisIgnat15, Ign12a,Ignat2012, KenigKoch2011}, and to the references found therein.
\par
In this paper we shall consider a rigidity property \`a la Liouville for divergence-free vector fields in $\R^n$ defined hereafter. 
\begin{defin}\label{defin:rigidity}
	Let $\phi:[0,+\infty)\to [0,+\infty)$ be a convex function such that $\phi(t)=0$ if and only if $t=0$. We say that the \textit{$\phi$-rigidity} property holds in $\R^n$ if, for any vector field $\eta = (\eta_1,\dots,\eta_n)\in L^\infty(\R^n;\R^n)$ such that
	\begin{itemize}
		\item[(i)] $\eta = 0$ on $\{x\in \R^n:\ x_n<0\}$;
		\item[(ii)] $\div \eta = 0$ on $\R^n$ in distributional sense;
		\item[(iii)] $\eta_n \ge \phi(|\eta|)$ almost everywhere on $\R^n$;
	\end{itemize}
	one has that $\eta \equiv 0$ almost everywhere on $\R^n$.
\end{defin}
We state here two results that establish this rigidity property.

%From a geometrical point of view, the vector fields that are considered for the linear rigidity belong to strictly convex cones having $\Rno$ as a supporting hyperplane. 

\begin{thm}[linear rigidity]\label{thm:rigidity1}
	Let $n\ge 2$. Then, the $\phi$-rigidity property holds in $\R^n$ when $\phi(t) = ct$ for some constant $c>0$.
\end{thm}

\begin{thm}[$\phi$-rigidity in the plane]\label{thm:rigidity2}
	Let $n= 2$. Then, the $\phi$-rigidity holds in $\R^2$ for any choice of $\phi$ as in Definition \ref{defin:rigidity}.
\end{thm}

Ideally, one would like to prove Theorem~\ref{thm:rigidity2} in any dimension. Yet, proving $\phi$-rigidity in dimension $n>2$ when $\phi$ is not linear is a rather delicate issue and it would require some further hypotheses. Indeed, we have found counterexamples for the choice $\phi(t)=ct^2$ in any dimension $n\ge 4$, as proved in Theorem~\ref{thm:controesempio}. At the current stage it is unclear to us if Theorem~\ref{thm:rigidity2} may hold in dimension $n=3$. Notice that the counterexamples found in dimension $n\ge 4$ are cylindrically symmetric; however, we prove in Proposition~\ref{prop:n=3} that no such vector field can be a counterexample for $n=3$.
\par
The specific choice $\phi(t)=t^2/2$ is quite interesting as it is closely related with a trace property of certain divergence-measure vector fields. More precisely, it allows us to deduce the existence of the trace of a divergence-measure vector field $\xi$ on an oriented, $\H^1$-rectifiable set $S \subset \R^2$, under a maximality assumption of the \textit{weak normal trace} of $\xi$ at $S$. In general, see Section~\ref{sec:wnt}, given a divergence-measure vector field $\xi$ defined on a bounded open set $\Omega\subset \R^n$, one is able to define its normal trace only in a distributional sense. More specifically, assuming that $\Omega$ is \emph{weakly-regular} (that is, the \textit{perimeter} of $\Omega$ is finite and coincides with the $(n-1)$-dimensional Hausdorff measure of $\de\Omega$, see Definition~\ref{def:weaklyreg}) one can show~\cite[Section 3]{SS16} that there exists a function $[\xi \cdot \nu_\Omega] \in L^\infty(\de \Omega; \H^{n-1})$ such that the following Gauss--Green formula holds for all $\psi \in C^1_c(\R^n)$
\begin{equation}\label{eq:GGintro}
\int_{\Omega} \psi(x)\, d \, \div \xi + \int_{\Omega} \xi(x)\cdot \nabla \psi(x)\, dx = \int_{\de \Omega} \psi(y)\, [\xi \cdot \nu_\Omega](y)\, d\H^{n-1}(y)\,,
\end{equation}
where we have denoted by $\H^{n-1}$ the $(n-1)$-dimensional Hausdorff measure and by $\nu_\Omega$ the measure-theoretic outer normal to $\de\Omega$. Such a function $[\xi \cdot \nu_\Omega]$ is called \emph{weak normal trace} of $\xi$ on $\de\Omega$. For the sake of completeness we recall that a first, fundamental weak version of the Gauss--Green formula is the classical result by De Giorgi~\cite{dG54, dG55} and Federer~\cite{Fed58}, which states that~\eqref{eq:GGintro} holds true in the case $\psi=1$, $\xi \in C^1$ and $\Omega$ with finite perimeter. A further extension due to Vol'pert~\cite{Volpert68,VolpertKhudyaev} holds when $\xi$ is weakly differentiable or BV. In the seminal works of Anzellotti~\cite{Anz83u, Anz83}, the concepts of weak normal trace and of \textit{pairing} between vector fields and (gradients of) functions are introduced for the first time. Since then, the class $\DM(\Omega)$ of divergence-measure, bounded vector fields has been widely studied in view of applications to hyperbolic systems of conservation laws~\cite{CF99, CT05, CTZ09}, to continuum and fluid mechanics~\cite{CF03}, and to minimal surfaces~\cite{LS17, SS16} among many others. In particular, the weak normal trace has been studied in different directions, see for instance~\cite{ACM05, ComiPayne, CT17, CdC18, CdCM19, CrastaDeCicco}.
\par
Despite some explicit characterizations of the weak normal trace are available (see the discussion in Section~\ref{sec:wnt}), a crucial issue coming with this distributional notion is that it is not possible to recover the pointwise value of such a trace by a standard, measure-theoretic limit, see Example~\ref{ex:wild}. However, assuming that $\|\xi\|_\infty \le 1$ and that the weak normal trace $[\xi\cdot \nu_S]$ attains the maximal value $1$ at some point $x\in S$, it would seem quite natural to expect that $\xi$ cannot oscillate too much around $x$, to ensure a maximal outflow at $x$. Thus, one is led to conjecture the existence of the classical trace, i.e. the validity of the formula  $[\xi\cdot \nu_S](x) = \nu_S(x)$. Indeed, this is exactly what we are able to prove, limitedly to the case $n=2$, by employing Theorem~\ref{thm:rigidity2} with the specific choice $\phi(t)=t^2/2$. 

\begin{thm}\label{thm:main}
	Let $\xi\in \DM(\R^2)$ and $S \subset \R^2$ an oriented $\H^{1}$-rectifiable with locally bounded $\H^{1}$-measure. Then, for $\H^1$-a.e. $x_0\in S$ such that $[\xi\cdot \nu_S](x_0) = \|\xi\|_\infty$ one has
	\[
	\aplim_{x\to  x_0^-} \xi(x) = \|\xi\|_\infty \ \nu_S(x_0)\,.
	\]
\end{thm}
In the above theorem the approximate limit is ``one-sided'' according to Definition~\ref{def:aplim}. Notice that the same conclusion of the theorem applies when the weak normal trace attains a local maximum for the modulus of the vector field, that is, when there exists an open set $U$ containing $x_0$ such that $[\xi\cdot \nu_S](x_0) \ge |\xi(x)|$ for almost every $x\in U$. Of course, by a scaling argument one can always restrict to vector fields $\xi$ with $\|\xi\|_\infty \le 1$.
\par 
The statement of Theorem~\ref{thm:main} holds for $\H^1$-a.e. $x_0\in S$. More precisely, one needs $x_0\in S$ to be such that: the normal of $S$ at $x_0$ is defined; it is a Lebesgue point for the weak normal trace; and $\div \xi$ does not concentrate around it. Asking $x_0$ to satisfy these hypotheses, makes us indeed discard an $\H^1$-negligible set of $S$.
\par
As the proof heavily relies on Theorem~\ref{thm:rigidity2}, we are able to show it only in dimension $2$. Despite there exist counterexamples to the quadratic rigidity in $\R^n$ when $n \ge 4$, we cannot exclude that Theorem~\ref{thm:main} might be true in any dimension. Were it false, one should be able to construct a vector field with maximal normal trace at some $(n-1)$-submanifold $S$, whose blow-ups at most points of $S$ are not unique. Indeed, the existence of the classical, one-sided trace of $\xi$ at $x_0$ is equivalent to the uniqueness of the blow-up of $\xi$ at $x_0$ (see Proposition~\ref{prop:zblow}). This essentially corresponds to a pointwise almost-everywhere convergence property. However, a maximality condition on the value of the weak normal trace  (viewed as a weak-limit of measures induced by classical traces on approximating smooth surfaces) might enforce no more than a $L^1$-type convergence, in analogy with what happens to a sequence of negative functions that weakly converge to zero. The fact that, in turn, $L^1$-convergence does not imply almost everywhere convergence (unless one extracts suitable subsequences) explains why proving Theorem~\ref{thm:main} is so delicate.
\par
Finally, we mention that Theorem~\ref{thm:main} allows us to strengthen the main result of our previous paper~\cite{LS17}. In there we considered solutions to the prescribed mean curvature equation in a set $\Omega$, in the so-called \textit{extremal case}. The extremality condition for the prescribed mean curvature equation
\[
\div \frac{\nabla u}{\sqrt{1+|\nabla u|^2}} = H\qquad \text{on $\Omega$}
\] 
is a critical situation for the existence of solutions, occurring when the following, necessary condition for existence 
\[
\left|\int_E H(x)\, dx\right| < P(E) \qquad \text{for all $E\subset\subset\Omega$}
\]
becomes an equality precisely at $E=\Omega$. When $\de\Omega$ is smooth, Giusti proved in his celebrated paper~\cite{Giusti1978} that the above necessary condition is also sufficient for existence, and moreover that the extremal case can be characterized by other properties, among which the uniqueness of the solution up to vertical translations and the vertical contact of its graph with the boundary of the domain. In the physically meaningful case, i.e. $\Omega\subset \R^2$, the extremal case for the prescribed mean curvature equation corresponds to the capillarity phenomenon for perfectly wetting fluids within a cylindrical container put in a zero gravity environment. 
\par
In our previous work~\cite{LS17} we extended Giusti's result to the wider class of weakly-regular domains, obtaining an analogous characterization that involves the weak normal trace of the vector field 
\[
Tu = \frac{\nabla u}{\sqrt{1+|\nabla u|^2}}
\]
on $\de\Omega$. In virtue of the result proved here, we obtain that the boundary behaviour of the capillary solution in the cylinder $\Omega\times \R$ actually improves to a vertical contact realized in the classical trace sense, rather than just in the sense of the weak normal trace. In other words, the normalized gradient $Tu$ is shown to admit a classical trace (equal to the outer normal $\nu_\Omega$) at $\H^1$-almost-every point of $\de\Omega$. We remark that, at present, no other technique, like the one based on regularity for almost-minimizers of the perimeter, seems to be applicable in the case of a generic weakly-regular domain. The reason is that the boundary of the cylinder $\Omega\times \R$ is not smooth enough to let the approaching boundary of the subgraph of the solution $u$ be uniformly regularized via the standard excess-decay mechanism for almost-minimizers of the perimeter. 
\par
Briefly, the paper is organized as follows. In Section~\ref{sec:prelim} we lay the notation, recall some basic facts from Geometric Measure Theory and weak normal traces.  In Section~\ref{sec:rigidity} we give the proofs of Theorem~\ref{thm:rigidity1} and of Theorem~\ref{thm:rigidity2}. Section~\ref{sec:counterexs} presents the construction of a counterexample to the $ct^2$-rigidity for $n\ge 4$. Finally,  Section~\ref{sec:traces} is devoted to the weak normal trace and to the proof of Theorem~\ref{thm:main}. 

\subsection*{Acknowledgements} The authors would like to thank Carlo Mantegazza for fruitful discussions on the topic, and Guido De Philippis for suggesting the ``flow-tube'' strategy used to prove Theorem~\ref{thm:rigidity1}.

%%%%%%%%%%%%%%%%%%%%%%%%%%%%%%
%%%%%%%%%%%%%%%%%%%%%%%%%%%%%%

\section{Preliminary notions and facts}\label{sec:prelim}
 
\subsection{Notation}
 We first introduce some basic notations. Given $n\ge 2$, $\R^{n}$ denotes the Euclidean $n$-dimensional space, $\Rnp$ the upper half-space $\{x\in \R^n:\ x_n>0\}$, and $\Rno$ the boundary of $\Rnp$. For any $x\in \R^{n}$ and $r>0$, $B_{r}(x)$ denotes the Euclidean open ball of center $x$ and radius $r$. Given a unit vector $v\in \R^n$ we set $B^{v}_r(z) = \{x\in B_r(z) \,:\, (x-z)\cdot v >0 \}$. Given a Borel set $E\subset \R^n$ we denote by $\chi_E$ its characteristic function, by $|E|$ its $n$-dimensional Lebesgue measure, and by $\H^{d}(E)$ its Hausdorff $d$-dimensional measure. We set $E_{x,r} = r^{-1}(E-x)$, where $E\subset \R^{n}$, $x\in\R^{n}$, and $r>0$. Given $\Omega\subset \R^{n}$ an open set, we write $E\subset\subset \Omega$ whenever the topological closure $\overline E$ of $E$ is a compact subset of $\Omega$. Whenever a measurable function, or vector field, $f$ is defined on $\Omega$, we denote by $\|f\|_{\infty}$ its $L^{\infty}$-norm on $\Omega$. We denote by $\DM(\Omega)$ the space of bounded vector fields defined in $\Omega$ whose divergence is a Radon measure. For brevity we set $\DM := \DM(\R^n)$. It is convenient to consider the restriction to $\DM(\Omega)$ of the weak-$*$ topology of $L^\infty$: given a sequence $\{v_k\}_k \subset \DM(\Omega)$, we say that $v_k$ converges to $v\in \DM(\Omega)$ in the weak-$*$ topology of $L^\infty$, and write $v_k\rightharpoonup v$ in $L^\infty$-$w^*$, if for every $f\in L^1(\Omega;\R^n)$ one has 
 \[
 \int_{\Omega} f \cdot v_k\, dx \to  \int_{\Omega} f \cdot v\, dx,\qquad \text{as }k\to \infty\,.
 \]
 
\subsection{Basic definitions in Geometric Measure Theory}

We now recall some basic definitions and facts from Geometric Measure Theory and, in particular, from the theory of sets of locally finite perimeter. 
 
 \begin{defin}[Points of density $\alpha$]\label{def:densityalpha}
Let $E$ be a Borel set in $\R^n$, $x\in \R^n$. If the limit
\[
\theta(E)(x) := \lim_{r\to0^+} \frac{|E\cap B_r(x)|}{|B_r(x)|}
\]
exists, it is called the density of $E$ at $x$. In general $\theta(E)(x) \in [0,1]$, hence, we define the set of points of density $\alpha\in [0,1]$ for $E$ as
\[
E^{(\alpha)} := \left\{ x\in \R^n\,:\, \theta(E)(x) = \alpha\right\}\,.
\]
\end{defin}

%\begin{defin}
%	Let $E\subset \R^n$ be a $\H^k$-measurable set. We say that $E$ 
%	\begin{itemize}
%		\item is countably $k$-rectifiable if there exist countably many Lipschitz maps $f_i :\R^k \to \R^n$ such that
%		\[
%		E \subset \bigcup f_i(\R^k)\,.
%		\]
%		\item is countably $\H^k$-rectifiable if there exist countably many Lipschitz maps $f_i :\R^k \to \R^n$ such that
%		\[
%		\H^k \left( E\setminus  \bigcup f_i(\R^k)\right) = 0\,.
%		\]
%		\item is $\H^k$-rectifiable if it is countably $\H^k$-rectifiable and $\H^k(E)<\infty$.
%	\end{itemize}
%\end{defin}

\begin{defin}[One-sided approximate limit]\label{def:aplim}
Let $\Omega\subset \R^n$ be an open set, and let $S\subset \Omega$ be an oriented $\H^{n-1}$-rectifiable set. Take $z \in S$ such that the exterior normal $\nu = \nu_S(z)$ is defined, and choose a measurable function, or vector field, $f$ defined on $\Omega$. We write
\[
\aplim_{x\to z^-} f(x) = w
\]
if for every $\alpha>0$ the set $\{x\in \Omega \cap B^{-\nu}_1(z):\ |f(x)-w|\ge \alpha\}$ has density $0$ at $z$.
\end{defin}
 
\begin{defin}[Perimeter]\label{def:per}
	Let $E$ be a Borel set in $\R^n$. We define the perimeter of $E$ in an open set $\Omega\subset \R^{n}$ as
	\[
	P(E; \Omega):=\sup \left\{ \int_\Omega \chi_E(x) \div \psi(x)\, dx\,: \psi\in C^1_c(\Omega;\, \R^n)\,, \|\psi\|_{\infty} \leq 1\right\}\,.
	\]
	We set $P(E) = P(E;\R^{n})$. If $P(E;\Omega)<\infty$ we say that $E$ is a set of finite perimeter in $\Omega$. In this case (see \cite{Mag12}) one has that the perimeter of $E$ coincides with the total variation $|D\chi_{E}|$ of the vector-valued Radon measure $D\chi_{E}$ (the distributional gradient of $\chi_{E}$), which is defined for all Borel subsets of $\Omega$ thanks to Riesz Theorem. We also recall that $P(E;\Omega) = \H^{n-1}(\de E\cap \Omega)$ when $\de E\cap \Omega$ is Lipschitz.
\end{defin}

\begin{thm}[De Giorgi Structure Theorem]\label{thm:DeGiorgi}
Let $E$ be a set of finite perimeter and let $\de^*E$ be the reduced boundary of $E$ defined as
\[
\de^* E:=\left \{x\in \de E\,:\, \lim_{r\to 0^+} \frac{D\chi_E(B_r(x))}{| D\chi_E|(B_r(x))} = -\nu_E(x) \in \mathbb{S}^{n-1} \right\}\,.
\]
Then,
\begin{itemize}
\item[(i)] $\de^{*}E$ is countably $\H^{n-1}$-rectifiable;

\item[(ii)] for all $x\in \de^{*}E$, $\chi_{E_{x,r}} \to \chi_{H_{\nu_E(x)}}$ in $L^{1}_{\text{loc}}(\R^{n})$ as $r\to 0^{+}$, where $H_{\nu_{E}(x)}$ denotes the half-space through $0$ whose exterior normal is $\nu_{E}(x)$;

\item[(iii)] for any Borel set $A$, $P(E;A) = \H^{n-1}(A\cap \de^{*}E)$, thus in particular $P(E)=\H^{n-1}(\partial^* E)$;

\item[(iv)] $\int_{E}\div \psi = \int_{\de^{*}E} \psi\cdot \nu_{E}\, d\H^{n-1}$ for any $\psi\in C^{1}_{c}(\R^{n};\R^{n})$.
\end{itemize}
\end{thm}

%\begin{thm}[Federer's Structure Theorem]\label{thm:Fed}
%Let $E$ be a set of finite perimeter. Then, $\de^* E \subset E^{(1/2)} \subset \de^e E$ and one has
%\[
%\H^{n-1}\left(\de^e E \setminus \de^* E\right)=0 \,.
%\]
%\end{thm}

Finally, we recall the notion of weakly-regular set which is useful for the next section.

\begin{defin}[Weakly-regular set]\label{def:weaklyreg}
An open, bounded set $\Omega$ of finite perimeter, such that $P(\Omega)=\H^{n-1}(\de \Omega)$, is said to be \emph{weakly-regular}.
\end{defin}

For the sake of completeness, we recall that examples of weakly-regular sets are, for instance, Lipschitz sets, or minimal Cheeger sets with $\H^{n-1}(\de \Om\cap \Om^{(1)}) = 0$ (see~\cite{LS17, Sar17} for an account of these facts,~\cite{Leo15, Par11, PS17} for an introduction to Cheeger sets, and~\cite{LNS17, LS17a} for recent results and examples).

\subsection{The weak normal trace and the Gauss--Green formula}\label{sec:wnt}

The weak normal trace was first defined in~\cite{Anz83} for a vector field with divergence in $L^1(\Omega)$, when $\Omega$ is bounded and Lipschitz. When $\Omega$ is a generic open set, and denoting by $\Leb^n$ the Lebesgue measure on $\R^n$, the weak normal trace $[\xi\cdot \nu_\Omega]$ can be defined as the distribution
\[
\langle [\xi\cdot \nu_\Omega], \psi \rangle :=
\int_{\Omega} \psi\, d\, \div \xi\ +\ \int_{\Omega} \xi \cdot \nabla \psi\, d\Leb^n,\qquad \psi\in C^\infty_c(\R^n)\,. 
\]
Taking $E\subset\subset \Omega$ of class $C^1$, and defining $[\xi\cdot \nu_E]$ in the same way, one can show that the distribution is represented by an $L^\infty$ function defined on $\de E$, that we still denote by $[\xi\cdot \nu_E]$ with a slight abuse of notation, so that 
\[
\langle [\xi\cdot \nu_E], \psi \rangle = \int_{\de E}\psi [\xi\cdot \nu_E]\, d\H^{n-1}\,.
\]
Then, up to showing a locality property of the weak normal trace on $C^1$ domains (see~\cite{ACM05}) it is possible to define $[\xi\cdot \nu_S]$ for any given, oriented $(n-1)$-rectifiable set $S$ contained in $\Omega$. We remark, however, that some particular care has to be taken when $\Omega$ is a bounded open set with finite perimeter (and $\xi$ is a-priori defined only on $\Omega$). In order to guarantee that the distributional weak normal trace is represented by an $L^\infty$ function defined on $S = \de\Omega$ one has to additionally assume that $\Omega$ is weakly-regular. 

Thus, when $\Omega$ is weakly-regular one obtains the Gauss--Green formula below (see~\cite[Section 3]{SS16}) exploiting the pairing between vector fields in $ \DM(\Omega)$ and functions in $C^1_c(\R^n)$. Another version of the formula with a slightly different pairing can be found in~\cite{LS17}. 
\begin{thm}[Generalized Gauss--Green formula]\label{thm:GG}
Let $\Omega\subset \R^n$ be a weakly-regular set. For any $\xi\in \DM(\Omega)$ and $\psi\in C^1_c(\R^n)$ one has
\begin{equation}\label{eq:GG}
\int_{\Omega} \psi\, d \div \xi + \int_{\Omega} \xi\cdot \nabla \psi\, dx = \int_{\de^*\Omega} \psi\, [\xi\cdot \nu_\Omega]\, d\H^{n-1}\,,
\end{equation}
where $[\xi\cdot \nu_\Om] \in  L^\infty(\de \Omega; \H^{n-1})$ is the so-called weak normal trace of $\xi$ on $\de^*\Omega$. 
\end{thm}
We remark for completeness that more general formulas have been recently obtained for pairings between $\xi\in \DM(\R^n)$, $\psi\in BV_{\text{loc}}(\R^n)\cap L^\infty_{\text{loc}}(\R^n)$ and $\Omega$ bounded with finite perimeter (see in particular~\cite{CrastaDeCicco,SS17}). 

As already observed by Anzellotti in~\cite{Anz83u} (see also~\cite{LS17}), the weak normal trace is a proper extension of the scalar product between the exterior normal $\nu_\Omega$ and the trace of the vector field $\xi$, assuming that the latter exists in the classical sense. More generally, one expects that the weak normal trace is obtained as a weak-limit of scalar products between (suitable averages of) the vector field $\xi$ and the normal field $\nu_S$. This actually corresponds to~\cite[Proposition 2.1]{Anz83u}, which says that, whenever $S$ is of class $C^1$ and $\div \xi \in L^1$, then 
\[
\frac{1}{\omega_n r^n}\int_{y\in B_r(\cdot)} \xi(y)\cdot \nu_S(\cdot)\, d\Leb^n(y)\to [\xi,\nu_S]\quad \text{in } L^\infty(S)-w^*\,,\quad \text{as $r\to 0^+$.}
\]
It is worth recalling that there exist also some pointwise characterizations of the weak normal trace. A first one is obtained by testing \eqref{eq:GG} with a function $\psi$ that approximates the characteristic function of a ball $B_r(x)$ for $x\in S$. Taking $x$ in a suitable subset of $S$ of full $\H^{n-1}$-measure one obtains
\[
[\xi,\nu_S](x) = \aplim_{r\downarrow 0}\ \frac{1}{\omega_{n-1}r^n} \int_{y\in \de B_r^-(x)} \xi(y) \cdot (x-y)\, d\H^{n-1}(y),\qquad \text{for $\H^{n-1}$-a.e. $x\in S$}\,,
\]
where, for $r>0$ small enough, $\de B_r^-(x)$ denotes the part of $\de B_r(x)$ which is on the side of $S$ pointed by $-\nu_S(x)$.  
A second characterization given in~\cite[Proposition 2.2]{Anz83u}, assuming $S$ of class $C^1$, is the following. Given $x\in S$ and $r,\rho>0$ small enough, let $S_\rho(x) = S\cap B_\rho(x)$ and set 
\[
Q_{r,\rho}(x) = \{z= y-t\nu_S(x):\ y\in S_\rho(x),\ t\in (0,r)\}\,,
\]
\[
\nu_{x}(z) = \nu_S(y)\qquad \text{if }z=y-t\nu_S(x)\,.
\]
Note that $\nu_x(z)$ is well-defined on $Q_{r,\rho}(x)$ as soon as $r,\rho$ are small enough. In other words, $Q_{r,\rho}(x)$ is a curvilinear rectangle foliated by translates of $S\cap B_\rho(x)$ in the $-\nu_S(x)$ direction. It is proved in~\cite{Anz83u} that 
\[
[\xi,\nu_S](x) = \lim_{\rho\downarrow 0} \lim_{r\downarrow 0} \frac{1}{\omega_{n-1}\rho^{n-1}r} \int_{z\in Q_{r,\rho}(x)} \xi(z)\cdot \nu_x(z)\, d\Leb^{n}(z)\,,
\]
for $\H^{n-1}$-almost-every $x\in S$. This result can be also obtained by testing~\eqref{eq:GG} with a $C^1$ function $\psi$ that satisfies $\psi=1$ on $S_\rho$, $\psi=0$ on $S_\rho - r\nu_S(x)$, and which is linear on any segment $(y,y-r\nu_S(x))$, $y\in S_\rho$. 
%In Proposition \ref{prop:caratraccia} we give a proof of these pointwise characterizations of the weak normal trace when $\xi\in \DM$ and $S$ is an oriented rectifiable set with finite $\H^{n-1}$ measure. \footnote{Penso sia possibile dimostrare entrambe le formule in questo caso piu' generale, basta adattare opportunamente le costruzioni di Anzellotti}.

Despite the characterizations described above, one can easily construct examples of vector fields like the one presented below (a variant of a piece-wise constant one defined in~\cite{Anz83u}) that illustrate possible wild behaviours of divergence-free vector fields for which the weak normal trace on $S$ is well-defined. More precisely, the example below shows that, in general, the weak normal trace does not coincide with any classical, or measure-theoretic, limit of the scalar product of the vector field with the normal to $S$, even when $S$ is of class $C^\infty$ and the vector field is divergence-free and smooth in a neighborhood of $S$ (minus $S$ itself). 
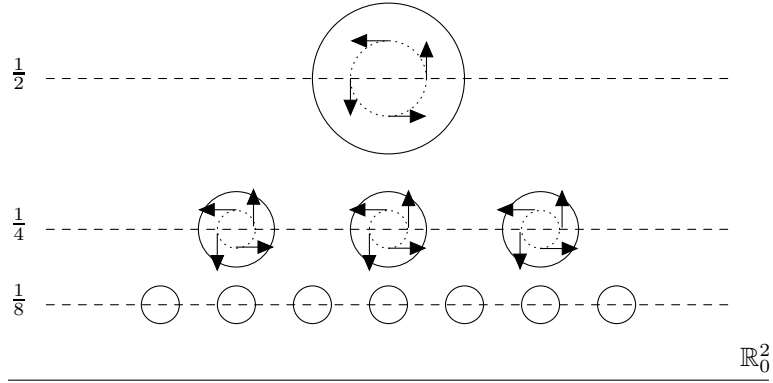
\begin{figure}[t]
	\centering
	\begin{tikzpicture}[line cap=round,line join=round,>=triangle 45,x=1.0cm,y=1.0cm]
	\clip(-5.5,-4.2) rectangle (5.5,2);
	\draw(0.,0.) circle (1.cm);
	\draw(-2.,-2.) circle (0.5cm);
	\draw(0.,-2.) circle (0.5cm);
	\draw(2.,-2.) circle (0.5cm);
	\draw(-3.,-3.) circle (0.25cm);
	\draw(-2.,-3.) circle (0.25cm);
	\draw(-1.,-3.) circle (0.25cm);
	\draw(0.,-3.) circle (0.25cm);
	\draw(1.,-3.) circle (0.25cm);
	\draw(2.,-3.) circle (0.25cm);
	\draw(3.,-3.) circle (0.25cm);
	\draw (4.5,-3.4) node[anchor=north west] {$\R^2_0$};
	\draw [line width=0.4pt,dash pattern=on 3pt off 3pt] (-4.5,0.)-- (4.5,0.);
	\draw [line width=0.4pt,dash pattern=on 3pt off 3pt] (-4.5,-2.)-- (4.5,-2.);
	\draw [line width=0.4pt,dash pattern=on 3pt off 3pt] (-4.5,-3.)-- (4.5,-3.);
	\draw [line width=0.4pt] (-5,-4.)-- (5,-4.);
	\draw (-5.1,0.4) node[anchor=north west] {$\frac12$};
	\draw (-5.1,-1.6) node[anchor=north west] {$\frac14$};
	\draw (-5.1,-2.6) node[anchor=north west] {$\frac18$};
	\draw [dotted] (0.,0.) circle (0.5cm);
	\draw [->,line width=0.4pt] (0.,0.5) -- (-0.5,0.5);
	\draw [->] (-0.5,0.) -- (-0.5,-0.5);
	\draw [->] (0.5,0.) -- (0.5,0.5);
	\draw [->] (0.,-0.5) -- (0.5,-0.5);
	\draw [dotted] (-2.,-2.) circle (0.25cm);
	\draw [dotted] (0.,-2.) circle (0.25cm);
	\draw [dotted] (2.,-2.) circle (0.25cm);
	\draw [->] (-2.0036990006562547,-2.235643179720274) -- (-1.503699000656255,-2.235643179720274);
	\draw [->] (0.,-2.25) -- (0.5,-2.25);
	\draw [->] (1.9975421684446364,-2.2546967090969448) -- (2.497542168444636,-2.2546967090969448);
	\draw [
	->] (0.2636709951675837,-1.9735020046534375) -- (0.26367099516758336,-1.4735020046534373);
	\draw [->] (2.283345109094699,-1.9735020046534366) -- (2.2833451090947,-1.4735020046534373);
	\draw [->] (-1.7750566481362058,-1.9544484752767661) -- (-1.7750566481362045,-1.4544484752767664);
	\draw [->,line width=0.4pt] (-2.0036990006562556,-1.7402514159268303) -- (-2.5036990006562556,-1.7402514159268303);
	\draw [->,line width=0.4pt] (-0.02213194548248045,-1.74025141592683) -- (-0.5221319454824804,-1.74025141592683);
	\draw [->,line width=0.4pt] (1.9213280509379527,-1.74025141592683) -- (1.4213280509379527,-1.74025141592683);
	\draw [->] (-2.251394882552978,-2.059553179123679) -- (-2.2513948825529773,-2.559553179123679);
	\draw [->] (-0.2507742980025313,-2.059553179123679) -- (-0.2507742980025314,-2.559553179123679);
	\draw [->] (1.7307927571712438,-2.0404996497470083) -- (1.7307927571712438,-2.5404996497470083);
	\end{tikzpicture}
	\caption{A twisting vector field defined on $\R^2_+$.}
	\label{EsQ}
\end{figure}
\begin{ex}\label{ex:wild}
	Let us set $\R^2_{+} =\{(x,y):\ y>0\}$, $S = \{(x,y):\ y=0\}$, and $\nu = (0,-1)$. 
	For $i\ge 1$ and $j=1,\dots,2^{i}-1$ we set $x_{ij} = \frac{j}{2^{i}}$, $y_{i} = \frac{1}{2^{i}}$, $r_{i}= \frac{1}{2^{i+2}}$. Then, for such $i$ and $j$ we take $f_{i}\in C^{\infty}_{c}(\R)$ with compact support in $(0,r_{i})$, so that in particular $f_{i}(0)=f_{i}(r_{i})=0$, and define $p_{ij} = (x_{ij},y_{i})$. Notice that by our choice of parameters, the balls $\{B_{ij} = B_{r_{i}}(p_{ij})\}_{i,j}$ are pairwise disjoint. Whenever $p\in B_{ij}$ we set
	\[
	\xi(p) = f_{i}(|p-p_{ij}|)(p-p_{ij})^{\perp}\,,
	\]
	while $\xi(p) = 0$ otherwise (see Figure~\ref{EsQ}). One can suitably choose $f_{i}$ so that $\|\xi\|_{L^{\infty}(B_{ij})}=1$ for all $i,j$. Moreover $\div \xi = 0$ on $\R^2_+$ and thus for any $\psi \in C^1_c(\R^2)$, by the Gauss--Green formula and owing to the definition of $\xi$, one has
\[
	\int_{S}\psi\, [\xi\cdot\nu]\, d\H^{1} = \int_{\R^2_+}\xi\cdot D\psi = \sum_{i,j}\int_{B_{ij}} \xi\cdot D\psi = \sum_{i,j}\int_{\de B_{ij}} \psi\, \xi\cdot \nu_{ij} = 0\,, 
\]
	so that $[\xi\cdot \nu] = 0$ on $S$. At the same time, $\xi$ twists in any neighborhood of any point $p_{0} = (x_{0},0)$, $x_{0}\in (0,1)$, and the second component of the average of $\xi$ on half balls centered at $p_0$ has a $\liminf$ strictly smaller than its $\limsup$ as $r\downarrow 0$. In conclusion, the scalar product $\xi(p)\cdot \nu(p_{0})$ does not converge to $0$ in any pointwise or measure-theoretic sense. 
\end{ex}
\medskip

%%%%%%%%%%%%%%%%%%%%%%%%%%%%%%
%%%%%%%%%%%%%%%%%%%%%%%%%%%%%%

\section{Proofs of the rigidity results}
\label{sec:rigidity}
%We recall the $\phi$-Rigidity Problem \ref{def:rigidity}. Let $\eta$ be a vector field in $L^\infty(\Rnp;\R^n)$. Let $\phi:[0,+\infty)\to [0,+\infty)$ be a continuous, strictly increasing and convex function, such that $\phi(0)=0$. Assume that 
%\begin{itemize}
%	\item[(i)] $\div \eta = 0$ on $\Rnp$;
%	\item[(ii)] $[\eta\cdot e_n]=0$ on $\Rno$;
%	\item[(iii)] $\eta_n \ge \phi(|\eta|)$ almost everywhere on $\Rnp$.
%\end{itemize}
%Then we ask whether it is true, or not, that one necessarily has $\eta = 0$. We observe that property (iii) could have been equivalently stated using a continuous, but not necessarily convex, function $\phi$, such that $\phi\ge 0$ and $\phi(t)=0$ if and only if $t=0$. Indeed, one can always replace $\phi$ by a smaller function which satisfies the same properties (i)-(iii) on some interval $[0,T]$ and, additionally, which is strictly increasing and convex on that interval.

We start by proving that the convexity assumption on the function $\phi$ makes Definition~\ref{defin:rigidity} essentially equivalent to a weaker one, in which the $C^\infty$ smoothness and the global Lipschitz-continuity of the vector field $\eta$ are required. 	
\begin{lem}\label{lem:mollifica}
Let $\phi$ and $\eta$ be as in Definition~\ref{defin:rigidity}. Let $\rho_\e$ be the standard mollifier supported in a ball of radius $\e>0$. Then, the regularized vector field $\eta^\e = \rho_\e\ast \eta$ is globally Lipschitz and satisfies (i)--(iii) of Definition~\ref{defin:rigidity} up to an $\e$-translation in the variable $x_n$.
\end{lem}
\begin{proof}
We note that $\eta_n \ge \phi(|\eta|)$ almost everywhere on $\R^n$ by (iii), therefore Jensen's inequality implies that for all $x\in \R^n$ and $\e>0$
\begin{align*}
\phi(|\eta^\e(x)|) &= \phi\left(\left| \int_{B_\e} \eta(y)\rho_\e(x-y)\,dy\right|\right) \le \int_{B_\e} \phi(|\eta(y)|) \rho_\e(x-y)\,dy\nonumber\\
&\le \int_{B_\e} \eta_n(y) \rho_\e(x-y)\, dy =  \eta^\e_{n}(x)\,.\label{eq:controllosueta}
\end{align*}
Then, we observe that $\eta^\e$ is globally Lipschitz, as a consequence of the boundedness of $\eta$, and verifies $\div \eta^\e = 0$ on $\R^n$, that is, (ii). Moreover, for every $x = (y,t)$ such that $t\le -\e$ one has $\eta^\e(x) = 0$. Finally, up to a translation in the variable $x$, we can assume $\eta^\e(x) = 0$ for all $x\in \Rnm$, hence property (i) is also satisfied. 
\end{proof}

%The proof of our first rigidity result, Theorem \ref{thm:rigidity1}, is presented hereafter.
\begin{proof}[Proof of Theorem~\ref{thm:rigidity1}]
Without loss of generality, by Lemma~\ref{lem:mollifica} we can assume that $\eta$ is smooth and globally Lipschitz. Let us fix $\e>0$ and consider the one-parameter flow associated with the vector field $X=\eta + \e e_n$, defined for every $t\in \R$ and $p\in \R^n$ by the Cauchy problem:
\begin{equation*}
\begin{cases}
\Phi(0,p) = p\,,\\
\frac{\de}{\de t} \Phi(t,p) = X(\Phi(t,p))\,.
\end{cases}
\end{equation*}
Note that in the notation above we dropped the dependence upon the parameter $\e$ for the sake of simplicity. Due to the smooth dependence from the initial datum, the map $\Phi(\cdot, \cdot):\R\times \R^n \to \R^n$ is smooth, $D_{p}{\Phi}_{|{t=0}} = \text{Id}$ and the map $\Phi(t, \cdot) : \R^n \to \R^n$ is a diffeomorphism for every $t\in \R$. Let us denote by $\Phi_n(t, p)$ the $n$-th component of $\Phi(t, p)$. 
Since $X_n(p) \ge \e$ for all $p\in \R^n$, we have that $\frac{\de}{\de t}\Phi_n(t, p) \ge \e$, hence the function $t\mapsto \Phi_n(t, p)$ is strictly monotone and surjective. Therefore, for every $h\in \R$ and $p\in \R^n$ there exists a unique $T=T(p,h)\in \R$ such that $\Phi_n(T, p) = h$. By the Implicit Function Theorem, $T(p,h)$ is smooth and one has 
\[
\de_h T(p,h) = X_n(\Phi(T, p))^{-1},\qquad D_p T(p,h) = -X_n(\Phi(T, p))^{-1} D_p \Phi_n(T, p)\,.
\]
Fix now an open, bounded and smooth set $A \subset \R^{n-1}$ and $h_0>0$. Let us set $A_{h_0}=A\times\{h_0\}$ and define the map 
\[
\Psi:A\times (0,h_0)\to \R^n,\qquad \Psi(q,h) = \Phi(T(p,h),p)\,,
\]
where we have set $p=(q,h_0)$. Before proceeding it is convenient to introduce some more notation. We write 
\begin{align*}
\Psi &= (\Psi^1,\dots,\Psi^{n-1},\Psi^n)=(\hat \Psi,\Psi^n)\\
X &= (X_1,\dots,X_{n-1},X_n) = (\hat X,X_n)\\
y &= (y_1,\dots,y_{n-1},y_n) = (\hat y, y_n)\qquad \text{for }y\in \R^n\,.
\end{align*}
We start by computing the partial derivative of $\Psi$ with respect to $h$:
\[
\de_h \Psi(q,h) = X_n(\Psi)^{-1} X(\Psi)\,.
\]
Note that $\Psi(p)=\Psi(q,h_0)=p$ and $\Psi^n(q,h)=h$ by definition. Owing to the smoothness of $\Psi$, we first compute the partial derivative with respect to $h$ of $\Psi^i_j := \de_{q_j}\Psi^i$, for $i,j=1,\dots,n-1$:
\begin{align*}
\de_h \Psi^i_j &= \de_{q_j} \de_h \Psi^i = \de_{q_j}[X_n(\Psi)^{-1} X_i(\Psi)]\\
&= X_n(\Psi)^{-1}\sum_{r=1}^n \de_{y_r}X_i(\Psi) \Psi^r_j - X_n(\Psi)^{-2}X_i(\Psi) \sum_{r=1}^n \de_{y_r} X_n(\Psi) \Psi^r_j\\
&= \sum_{r=1}^n \left[X_n(\Psi)^{-1}\de_{y_r}X_i(\Psi) - X_n(\Psi)^{-2}X_i(\Psi) \de_{y_r} X_n(\Psi)\right] \Psi^r_j\,.
\end{align*}
Moreover, it is immediate to check that $\de_{q_j} \Psi^n(q,h) = 0$ for all $j=1,\dots,n-1$ and that $\de_h \Psi^n(q,h)=1$, so that in particular the matrix form of the previous computation is
\[
\de_h D_q\hat{\Psi} = B(\Psi) \cdot D_q\hat{\Psi}\,,
\]
where 
\[
B= X_n^{-1}\, D_{\hat{y}}\hat{X} - X_n^{-2}\, \hat{X}\otimes D_{\hat{y}} X_n\,.
\]
Moreover, the determinant of $D\Psi$ coincides with that of $D_q\hat{\Psi}$. If we assume that $q\in \R^{n-1}$ is fixed, and define $\delta(h) = \det D_q\hat{\Psi}(q,h)$, we find by standard calculations that 
\[
\delta^\prime(h) = \mathop{\mathrm{tr}}\,[B(\Psi(q,h))]\, \delta(h)\,,
\]
with initial condition $\delta(h_0) = 1$. This shows that the Jacobian matrix of $\Psi$ is uniformly invertible on compact subsets of $\R^{n}$. Consequently, $\Psi$ is a smooth diffeomorphism and the set $F_\e(A) := \Psi(A\times (0,h_0))$, depicted in Figure~\ref{fig:flusso_cauchy}, is Lipschitz. Notice moreover that 
\[
|\Psi(q,h_0) - \Psi(q,0)| \le \int_0^{h_0}|\de_h \Psi(q,h)|\, dh \le \int_0^{h_0}X_n^{-1}|X|\, dh \le h_0 (c|\eta|+\e)^{-1}(|\eta|+\e) \le \max\left\{h_0, \frac{h_0}{c}\right\}\,,
\]
thanks to the properties of $\eta$. 
\begin{figure}[t]
 \includegraphics[scale=0.75]{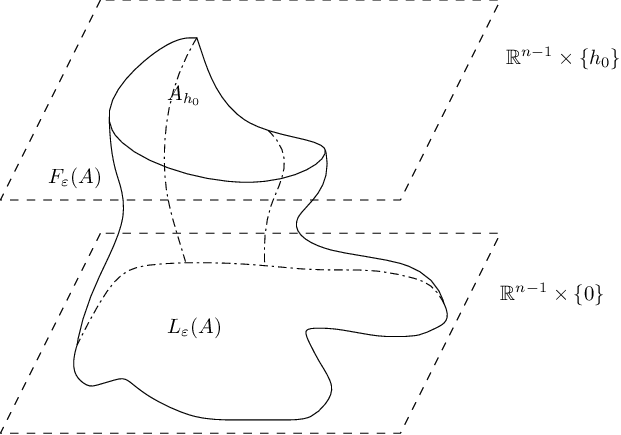}
  \caption{\label{fig:flusso_cauchy} The evolution of $A_{h_0}$ through the diffeomorphism $\Psi$ defines a  ``flow-tube'' .}
  \end{figure}
Notice now that given a bounded $A\subset \R^{n-1}$, there exists a constant $R>0$ depending only on $A$ and $h_0$, such that $F_\e(A)$ is contained in $(-R,R)^n$ for every $\e>0$. Setting $L_\e(A) = \Psi(A\times \{0\})$, by applying the Divergence Theorem on $F_\e(A)$ to the vector field $X = \eta + \e e_n$ we find the identity
\[
0 = \int_{F_\e(A)} \div X = \int_A (\eta_n(q,h_0) + \e)\, dq - \e\H^{n-1}(L_\e(A))\,,
\]
since the integral on the ``lateral'' boundary of $F_\e(A)$ vanishes. This happens because the lateral boundary of the flow-tube consists of integral curves of the flow $X$, and thus on this lateral boundary one has $X\cdot \nu_{F_\e(A)}\equiv 0$. Using the fact that $\H^{n-1}(L_\e(A)) \le (2R)^{n-1}$ we can pass to the limit as $\e\to 0$ in the identity above, obtaining that $\eta_n$ vanishes on $A\times \{h_0\}$. By the arbitrary choice of both $A$ and $h_0$, and by property (iii) of Definition~\ref{defin:rigidity}, we conclude that $\eta=0$ on $\R^n$.
\end{proof}

We now proceed with the proof of Theorem~\ref{thm:rigidity2}. Here, instead of using the ``flow-tube'' method employed in the proof of Theorem~\ref{thm:rigidity1}, we take advantage of a special feature of $2$-dimensional cylinders of the form $(-a,a)\times(0,h_0)$, i.e.~the fact that their ``lateral'' perimeter is constantly equal to $2h_0$ (and in particular it does not blow up when $a\to +\infty$).
\begin{proof}[Proof of Theorem~\ref{thm:rigidity2}]
By Lemma~\ref{lem:mollifica} we can additionally assume that $\eta$ is smooth and globally Lipschitz. Since $\eta_2 \ge 0$, by the Divergence Theorem we find
\begin{equation}\label{eq:dominated}
\int_{-r}^r |\eta_2(x_1,t)|\, dx_1 = \int_{-r}^r \eta_2(x_1,t)\, dx_1 = \int_0^t \eta_1(-r, x_2) - \eta_1(r, x_2)\, dx_2\, \le 2t\|\eta\|_\infty\,.
\end{equation}
Therefore, $\eta_2(\cdot, t) \in L^1(\R;\R)$ for all $t>0$ and
\begin{equation}\label{eq:normz2}
\|\eta_2(\cdot, t)\|_1 \le 2t\|\eta\|_\infty\,.
\end{equation}
By combining (iii) of Definition~\ref{defin:rigidity} with~\eqref{eq:normz2} and the fact that $\eta_{2}$ is Lipschitz, we infer that 
\begin{equation}\label{eq:ubonz1}
\phi\Big(|\eta_1(x_1, t)|\Big) \le \eta_2(x_1, t) \to 0\qquad \text{as $x_1\to \pm\infty$,}
\end{equation}
hence, owing to the properties of $\phi$, we get for all $t>0$
\begin{equation}\label{eq:limitz2e}
\lim_{x_1\to\pm \infty} \eta_1(x_1, t) = 0\,.
\end{equation}
This implies that 
\[
2\|\eta\|_\infty \ge \eta_1(-r,t) - \eta_1(r,t) \to 0
\] 
as $r \to +\infty$, for all $t>0$. Therefore, we can take the limit in~\eqref{eq:dominated} as $r\to \infty$ and, thanks to the Dominated Convergence Theorem, we obtain that
\[
\lim_{r\to +\infty} \int_{-r}^r |\eta_2(x_1, t)|\, dx_1 = \lim_{r\to +\infty}  \int_0^t \eta_1(-r, x_2) - \eta_1(r, x_2)\, dx_2 =  \int_0^t \lim_{r\to +\infty} (\eta_1(-r, x_2) - \eta_1(r, x_2))\, dx_2 = 0\,.
\]
Hence, the $L^1$-norm of $\eta_2(\cdot, t)$ is zero, thus $\eta_2(\cdot, t)=0$, for all $t>0$. By~\eqref{eq:ubonz1} and the properties of $\phi$ we get $\eta=0$ on $\R^2$.
\end{proof}

%%%%%%%%%%%%%%%%%%%%%%%%%%%%%%
%%%%%%%%%%%%%%%%%%%%%%%%%%%%%%

\section{Counterexamples to quadratic rigidity in dimension \texorpdfstring{$n\geq 4$}{}}
\label{sec:counterexs}

Given $x\in \R^n$ we set $r = (x_1,\dots,x_{n-1})$ and $z = x_n$.  We consider vector fields $\eta\in C^0(\R^n)$ such that for $r\neq 0$ one has
\begin{equation}\label{eq:etaspecial}
\eta(r,z) = \big(r f(r,z),\ h(r,z) \big)\,,
\end{equation}
with $f,h\in C^1(\R^n_+ \setminus \{(0,z):\ z\in \R\})$. We have the following 
\begin{prop}\label{prop:equivalence}
	Let $\eta\in C^0(\R^n)$ be as in \eqref{eq:etaspecial}. Define
	\begin{equation}\label{eq:Vfrometa}
	V(r,z):= -|r|^{n-1}  \int_0^z f(r,s)\, ds\,.
	\end{equation}
	Then, $\eta$ satisfies 
	\begin{align}\label{eq:etazero}
	&\eta(r,z) = 0\qquad \text{for all $z\le 0$,}\\\label{eq:etaminoreuno}
	&|\eta(x)|\le 1\qquad \text{for all }x\in \R^n,\\\label{eq:divetazero}
	&\div \eta = 0 \qquad \text{on $\R^n$},\\\label{eq:t2rigidity}
	&\eta_n(x) \ge c_1|\eta(x)|^2\qquad \text{for all $x\in \R^n$,}
	\end{align}
	if and only if $V(r,z)$ satisfies
	\begin{align}\label{eq:V0}
	&V(r,z)=0 \qquad\text{for all $z\le 0$},\\\label{eq:V1}
	&|\nabla V(r,z)|\le |r|^{n-2} \qquad\text{for all $r>0$ and $z\in \R$},\\\label{eq:V2}
	& r\cdot \nabla_r V(r,z) \ge c_2 |r|^{3-n}\big(\de_z V(r,z)\big)^2, \qquad\text{for all $r>0$ and $z\in \R$}\,.
	\end{align}
	Moreover one has 
	\begin{equation}\label{eq:etafromV}
	\eta(r,z) = |r|^{1-n} \Big(- (\de_z V) r,r\cdot \nabla_r V\Big)
	\end{equation}
\end{prop}
\begin{proof}
	We show in full detail the only if part. Since
\[
	\div \eta (r,z) = \div_r (r f(r,z))+ \de_z h(r,z) 
	= (n-1)f + r\cdot \nabla_r f(r,z) + \de_z h(r,z)\,,
\]
	equation~\eqref{eq:divetazero} is equivalent to
	\begin{equation}\label{eq:eq1}
	\de_z h = -\left( r\cdot \nabla_r f + (n-1)f\right)\,.
	\end{equation}
	By recalling that $h(r,0) = 0$ for all $r$ by~\eqref{eq:etazero}, from~\eqref{eq:eq1} we obtain that
	\begin{equation}\label{eq:h-integrale}
	h(r,z)=-\int_0^z  \big( r\cdot \nabla_r f(r,s) + (n-1)f(r,s) \big)\, ds\,.
	\end{equation}
	Inequality~\eqref{eq:t2rigidity}, up to a change of the constant $c$, is equivalent to 
	\begin{equation}\label{eq:eq2}
	h \ge c|r|^2 f^2\,.
	\end{equation}
	Let us set $V$ as in~\eqref{eq:Vfrometa} and observe that~\eqref{eq:etazero} implies $V(r,z)=0$ when $z\le 0$. Then using~\eqref{eq:h-integrale} we obtain
	\[
	\nabla_{r}V = - |r|^{n-1}\int_0^z [\nabla_r f(r,s) + (n-1)f(r,s)|r|^{-2}r]\, ds
	\]
	and 
	\begin{equation}\label{gradV1}
	\de_z V(r,z) = -|r|^{n-1}f(r,z)\,,
	\end{equation}
	so that in particular 
	\begin{equation}\label{gradV2}
	r\cdot \nabla_r V = |r|^{n-1}h\,,
	\end{equation}
	and the inequality~\eqref{eq:eq2} implies
	\begin{equation*}
	r\cdot \nabla_r V \ge c|r|^{n+1}f^2 \ge c|r|^{3-n} (\de_z V)^2\,.  
	\end{equation*}
	Then by observing that~\eqref{eq:etaminoreuno} is equivalent to $|r|^2 f^2 + h^2\le 1$, we obtain by~\eqref{gradV1} and~\eqref{gradV2}
\[
	\label{cond1}
	|\nabla V| \le |r|^{n-2}\,. 
\]
	We have thus proved that defining $V$ as in~\eqref{eq:Vfrometa} we get~\eqref{eq:V0}--\eqref{eq:V2}.
	\par
	Conversely, one can easily check that given $V$ satisfying~\eqref{eq:V0}--\eqref{eq:V2}, the vector field $\eta$ defined as in~\eqref{eq:etafromV} satisfies~\eqref{eq:etazero}--\eqref{eq:t2rigidity}.
\end{proof}

\begin{thm}\label{thm:controesempio}
	The quadratic rigidity property does not hold in dimension $n\ge 4$.
\end{thm}	
\begin{proof}
	Let us consider a positive parameter $\gamma$ (to be chosen later) and define the function
	\[
	V(r,z) = \begin{cases}
	\gamma [(1+|r|^{n-1})^{\frac{1}{n-1}}-1]\arctan(z^2) & \text{if }z\ge 0\,,\\
	0 & \text{otherwise}\,.
	\end{cases}
	\]
	Our aim is to verify properties~\eqref{eq:V0}--\eqref{eq:V2} up to a suitable choice of $\gamma$, and then to use the equivalence stated in Proposition~\ref{prop:equivalence}. Of course~\eqref{eq:V0} is true by definition of $V$. 
	Let us set $\rho = |r|$ and write $V=V(\rho,z)$ for simplicity. One has
	\[
	\de_\rho V = \gamma \arctan(z^2) (1+\rho^{n-1})^{\frac{2-n}{n-1}} \rho^{n-2}
	\]
	and
	\[
	\de_z V = 2\gamma [(1+\rho^{n-1})^\frac{1}{n-1} -1] \frac{z}{1+z^4}\,.
	\]
	Notice that $V$ is of class $C^1$ and $V(\rho,z) = 0$ for all $\rho>0$ and $z\le 0$. We obtain
	\begin{align}
	|\nabla V| &\nonumber\le \gamma \arctan(z^2) \big(1+\rho^{n-1}\big)^{\frac{2-n}{n-1}} \rho^{n-2} + 2\gamma \Big[\big(1+\rho^{n-1}\big)^{\frac{1}{n-1}} -1\Big] \frac{z}{1+z^4}\\\label{stimagradV1}
	&\le \gamma\frac{\pi\rho^{n-2}}{2} + \gamma\frac{3^{3/4}}{2}\Big[\big(1+\rho^{n-1}\big)^{\frac{1}{n-1}} -1\Big]\,, 
	\end{align}
	where the second inequality follows from the maximization of the function $\frac{z}{(1+z^4)}$ for $z\in [0,+\infty)$.
	Let us consider the function
	\[
	\phi(\rho) = \frac{(1+\rho^{n-1})^\frac{1}{n-1} -1}{\rho^{n-2}}\,.
	\]
	As $\rho\to 0^+$ we have $\phi(\rho) \simeq \rho/(n-1)$, while as $\rho\to +\infty$ we have $\phi(\rho) \simeq \rho^{3 -n}$. Moreover, when $0<t<1$, one has $(1+t)^\frac{1}{n-1}\le 1+\frac{t}{n-1}$, hence we deduce that 
	\[
	\phi(\rho) \le \frac{\rho}{n-1} \le 1
	\]
	when $\rho <1$, and  
	\[
	\phi(\rho) \le \rho^{3 -n} \le 1
	\] 
	when $\rho\ge 1$. Therefore by~\eqref{stimagradV1} and the last inequalities we find that
	\begin{equation}
	|\nabla V| \le C\gamma \rho^{n-2}\,,
	\end{equation}
	where 
	\[
	C = \frac{\pi+3^{3/4}}{2}\,.
	\]
	Assuming $\gamma \le C^{-1}$ we obtain~\eqref{eq:V1}.
	
	We now show that~\eqref{eq:V2} (with the constant $c=1$) holds up to taking a smaller $\gamma$. Indeed, the relation $\rho \de_\rho V \ge \rho^{3-n} (\de_z V)^2$, after separation of variables, becomes
	\[
	\frac{\arctan(z^2)(1+z^4)^2}{z^2}  \ge 4\gamma \phi(\rho)^2(1+\rho^{n-1})^{\frac{n-2}{n-1}}\,.
	\]
	We argue as for the upper bound of $\phi(\rho)$ (more precisely, we discuss the two cases $\rho<1$ and $\rho\ge 1$; in the first case we use the bound $\phi(\rho)\le 1$, while in the second case we use the fact that $n\ge 4$, and the inequalities $\phi(\rho) \le \rho^{3-n}$ and $1+\rho^{n-1}\le 2\rho^{n-1}$), and find
	\begin{equation}\label{eq:stimadim4}
	4\gamma \phi(\rho)^2(1+\rho^{n-1})^{\frac{n-2}{n-1}}\le 2^{\frac{3n-4}{n-1}} \gamma\,.
	\end{equation}
	At the same time, by easy calculations we infer that the function $\frac{\arctan(t) (1+t^2)^2}{t}$ is bounded from below by $1$. Hence~\eqref{eq:V2} is implied by the condition
	\[
	\gamma \le 2^{\frac{4-3n}{n-1}}\,.
	\]
	In conclusion, by taking $\gamma$ small enough, and thanks to Proposition~\ref{prop:equivalence}, a divergence-free vector field providing a counterexample to the rigidity property in dimension $n\ge 4$ is given by
	\[
	\eta(r,z) = \begin{cases}
	\gamma|r|^{1-n}\left(-2[(1+|r|^{n-1})^\frac{1}{n-1} -1] \frac{z}{1+z^4}\, r,\  \arctan(z^2) (1+|r|^{n-1})^{\frac{2-n}{n-1}} |r|^{n-1}\right) & \text{if }z>0,\\
	0& \text{if }z\le 0\,.
	\end{cases}
	\]
	We remark that the vector field $\eta$ is of class $C^0$, however one can obtain a $C^\infty$ counterexample by mollification of $\eta$.
\end{proof}

Concerning the $3$-dimensional case, the construction of a counterexample with cylindrical symmetry, as done in Theorem~\ref{thm:controesempio}, does not work. Indeed we are unable to get an estimate like~\eqref{eq:stimadim4}, as the function $\phi(\rho)$ tends to $1$ as $\rho\to+\infty$. More precisely we can prove the following result.
\begin{prop}\label{prop:n=3}
	Let $n=3$ and assume that $V(r,z)$ is a $C^1$ function satisfying properties~\eqref{eq:V0}--\eqref{eq:V2}. Then, $V(r,z) = \psi_1(|r|) \psi_2(z)$ for suitable functions $\psi_1,\psi_2$ implies $V\equiv 0$.
\end{prop}
\begin{proof}
	We set $\rho = |r|$ and write~\eqref{eq:V2} as
	\[
	\rho \psi_1'(\rho)\psi_2(z) \ge c \psi_1^2(\rho) [\psi_2'(z)]^2\,.
	\]
	Let us assume by contradiction that $\psi_1$ and $\psi_2$ are not trivial, hence there exist $z_0>0$ and $\rho_0>0$ such that $\psi_2(z_0)\neq 0$, $\psi_2'(z_0)\neq 0$, and $\psi_1(\rho_0)\neq 0$. Setting $a=\psi_2(z_0)$ and $b=\psi_2'(z_0)$, we have four cases according to the sign of $a$ and $\psi_1(\rho_0)$. 
	We discuss the first case $a>0$ and $\psi_1(\rho_0)>0$. We consider the differential inequality
	\[
	\rho\psi_1'(\rho) \ge \gamma \psi_1^2(\rho)
	\]
	with $\gamma = cb^2/a>0$. Therefore, the function $\psi_1$ is increasing and by separation of variables and integration between $\rho_0$ and $\rho>\rho_0$ we get
	\begin{equation}\label{eq:psiuno}
	\psi_1^{-1}(\rho_0) \ge -\psi_1^{-1}(\rho) + \psi_1^{-1}(\rho_0) \ge \gamma \log(\rho/\rho_0)
	\end{equation}
	We denote by $I=[\rho_0,\beta)$ the maximal right interval of existence of the solution $\psi_1$, for which $\psi_1>0$. We can exclude the case $\beta<+\infty$, as we would obtain by maximality that $\psi_1(\rho)\to +\infty$ as $\rho\to \beta^-$, however this would contradict the fact that $|\psi'(\rho)|\le \rho$ for all $\rho>0$. Contrarily, if $\rho\to+\infty$ one gets a contradiction with~\eqref{eq:psiuno}. The remaining three cases can be discussed in a similar way. 
\end{proof}

\begin{rem}
	As a consequence of Proposition~\ref{prop:n=3} we infer that in dimension $n=3$ no counterexample to the rigidity property can be found in the class of vector fields of the form $\eta(r,z) = \big(rf(|r|,z),h(|r|,z)\big)$.
\end{rem}

%%%%%%%%%%%%%%%%%%%%%%%%%%%%%%
%%%%%%%%%%%%%%%%%%%%%%%%%%%%%%

\section{The trace of a vector field with locally maximal normal trace}
\label{sec:traces}

The results we shall discuss in this section are stated for $\H^{n-1}$-almost-every point of $S$, being $S$ an oriented, $\H^{n-1}$-rectifiable set with locally finite $\H^{n-1}$-measure. Therefore, given $z\in \DM$, without loss of generality (see~\cite[Theorem 2.56]{AFP00}) we shall assume $x_0\in S$ to be such that 
\begin{itemize}%\label{sec:treabc}
\item[(a)] the normal vector $\nu_S(x_0)$ is defined at $x_0$;
\item[(b)] $x_0$ is a Lebesgue point for the weak normal trace $[z\cdot \nu_S]$ of $z$ on $S$, with respect to the measure $\H^{n-1}\llcorner S$;
\item[(c)] $|\div z|(B_r(x_0)\setminus S) = o(r^{n-1})$ as $r\to 0^+$.
\end{itemize}

\begin{lem}\label{lem:divzero}
Let $\eta,\{z_k\}_k$ be vector fields in $\DM$ with $\sup_k \|z_k\|_\infty <+\infty$. Let $\Sigma,\{S_k\}_k$ be oriented, closed $\H^{n-1}$-rectifiable sets with locally finite $\H^{n-1}$-measure, satisfying the following properties:
\begin{itemize}
	\item[(i)] $z_k\to \eta$ in $L^\infty$-$w^*$;
	\item[(ii)] $\H^{n-1}{\llcorner S_k} \rightharpoonup \H^{n-1}{\llcorner \Sigma}$;
	\item[(iii)] $|\div z_k|\llcorner(\R^n\setminus S_k) \rightharpoonup 0$.
\end{itemize}  
Then, $\div \eta = 0$ in $\R^n\setminus \Sigma$.
\end{lem}

\begin{proof}
Fix $\phi \in C^1_c(\R^n\setminus \Sigma)$ and set $\mu_k = \div z_k$ and $\mu = \div \eta$. By the Divergence Theorem coupled with the formula (see~\cite[Proposition 3.2]{ACM05})
\[
\mu_k \llcorner S_k = \Big([z_k\cdot \nu_{S_k}]^+ - [z_k\cdot \nu_{S_k}]^-\Big)\, \H^{n-1}\llcorner S_k\,,
\]
we have
\begin{align*}
\int_{\R^n\setminus S_k} \phi\, d\mu_k &= -\int_{S_k} \phi\, d\mu_k - \int_{\R^n\setminus S_k} \nabla \phi \cdot z_k \\
&= -\int_{S_k} \phi \Big([z_k\cdot \nu_{S_k}]^+ - [z_k\cdot \nu_{S_k}]^-\Big)\, d\H^{n-1} - \int_{\R^n} \nabla \phi \cdot z_k\,,
\end{align*}
hence by (ii) and (iii)
\[
\left|\int_{\R^n} \nabla \phi \cdot z_k \right| \le \int_{\R^n\setminus S_k}| \phi |\, d\mu_k + 2\|z_k\|_\infty \int |\phi|\, d\,\H^{n-1}\llcorner S_k \to 0\qquad \text{as }k\to\infty\,.
\]
This shows that
\[
\int_{\R^n} \phi\, d\mu = \lim_k\int_{\R^n} \phi\, d\mu_k = \lim_k \int_{\R^n} \nabla \phi \cdot z_k = 0\,,
\]
which proves the thesis.
\end{proof}

\begin{prop}\label{prop:zblow}
Let $z \in \DM$ and let $S$ be a closed, oriented $\H^{n-1}$-rectifiable set with locally finite $\H^{n-1}$-measure. Then, for $\H^{n-1}$-almost-every $x_0\in S$ and for any decreasing and infinitesimal sequence $\{r_k\}_k$, the sequence $z_k$ of vector fields defined by $z_k(y) = z(x_0+r_k y)$ converges up to subsequences to a vector field $\eta \in \DM$ in $L^\infty$-$w^*$, such that setting $\Sigma = [\nu_S(x_0)]^\perp$, we have $\div \eta = 0$ on $\R^n\setminus \Sigma$ and $[\eta\cdot \nu_{\Sigma}] = [z\cdot \nu_S](x_0)$ on $\Sigma$. 
\end{prop}

\begin{proof}
We show that hypotheses~(i)--(iii) of Lemma \ref{lem:divzero} are satisfied. 
Since $\|z_k\|_\infty = \|z\|_\infty$ for all $k$, thanks to Banach-Alaoglu Theorem (see also~\cite[Theorem 3.28]{Bre11}) we can extract a not relabeled subsequence converging to $\eta\in \DM$, which gives~(i). We set $S_k = r_k^{-1}(S - x_0)$, then thanks to~(c) we have
\begin{equation}\label{eq:divzkosmall}
|\div z_k|(B_R\setminus S_k) = r_k^{1-n} |\div z|(B_{Rr_k}(x_0) \setminus S) \to 0 \qquad \text{as }k\to\infty\,,
\end{equation}
for all $R>0$, which gives~(iii). Owing to the localization property proved in~\cite[Proposition 3.2]{ACM05} we can replace $S$ with the boundary of an open set $\Omega$ of class $C^1$, such that $x_0\in \de\Omega$ and $\nu_S(x_0) = \nu_\Omega$. Defining $\Omega_k = r_k^{-1}(\Omega - x_0)$, the proof of~(ii) is reduced to showing that $\H^{n-1}\llcorner \de\Omega_k$ weakly converge as measures to $\H^{n-1}\llcorner \de H$, where $H$ is the tangent half-space to $\Omega$ at $x_0$ (so that $\Sigma = \de H$). This fact is a consequence of Theorem~\ref{thm:DeGiorgi}. Now we can apply Lemma~\ref{lem:divzero} and obtain $\div \eta = 0$ on $\R^n\setminus \Sigma$. In order to prove the last part of the statement we have to show that for any $\psi\in C^1_c(\R^n)$ one has 
\[
\int_H \psi\, d\, \div \eta + \int_H \nabla\psi \cdot \eta = \tau_0\, \int_{\de H} \psi\, d\H^{n-1}\,,
\]
where $\tau_0 = [z\cdot \nu_S(x_0)]$. Since we have proved that $\div \eta = 0$ on $\R^n\setminus \Sigma$, we only have to show that 
\begin{equation}\label{eq:intHphieta}
\int_H \nabla\psi \cdot \eta = \tau_0\, \int_{\de H} \psi\, d\H^{n-1}\,.
\end{equation}
Note that by the convergence of $z_k$ to $\eta$, the $L^1_{\text{loc}}$-convergence of $\Omega_k$ to $H$, and the convergence of the measures $\H^{n-1}\llcorner \de\Omega_k$ to $\H^{n-1}\llcorner \de H$, we have
\begin{equation}
\int_H \nabla\psi \cdot \eta - \tau_0 \int_{\de H}\psi\, d\H^{n-1} = \lim_k \int_{\Omega_k} \nabla\psi \cdot z_k - \tau_0\int_{\de \Omega_k}\psi\, d\H^{n-1}\,.\label{eq:limversoH}
\end{equation}
Therefore by~\eqref{eq:divzkosmall} and~(b) we infer that
\begin{align*}
\left|\int_{\Omega_k}\nabla \psi \cdot z_k - \tau_0\int_{\de\Omega_k} \psi\, d\H^{n-1}\right| 
&\le
\left|\int_{\de\Omega_k} \psi \, \Big([z_k\cdot \nu_{\Omega_k}] - \tau_0\Big)\, d\H^{n-1} - \int_{\Omega_k} \psi\, d\, \div z_k\right|\\ 
&\le 
 \|\psi\|_\infty\, \int_{\de\Omega_k\cap \spt \psi} |[z_k\cdot \nu_{\Omega_k}] - \tau_0| +
  \int_{\R^n\setminus \de \Omega_k} |\psi|\, |\div z_k|\\
  & \to 0,\qquad \text{as }k\to\infty\,.
\end{align*}
Combining this last fact with~\eqref{eq:limversoH} implies~\eqref{eq:intHphieta} at once, and concludes the proof.
\end{proof}

By relying on Proposition~\ref{prop:zblow} and on Theorem~\ref{thm:rigidity2}, we are now able to prove the main result of the section, i.e.~Theorem~\ref{thm:main} which states the existence of the classical trace for a divergence-measure vector field having a maximal weak normal trace on a oriented $\H^1$-rectifiable set $S$.

%For sake of clarity we recall the statement below.
%\begin{thm*}
%	Let $\xi\in \DM(\R^2)$ and $S \subset \R^2$ an oriented $\H^{1}$-rectifiable with locally bounded $\H^{1}$-measure. Then, for $\H^1$-a.e. $x_0\in S$ such that $[\xi\cdot \nu_S](x_0) = \|\xi\|_\infty$ one has
%	\begin{equation*}
%	\aplim_{x\to  x_0^-} \xi(x) = [\xi\cdot\nu_S](x_0)\ \nu_S(x_0)\,.
%	\end{equation*}
%\end{thm*}

\begin{proof}[Proof of Theorem~\ref{thm:main}]
Without loss of generality, up to a translation we can suppose $x_0 =0$ and up to a rotation that $\nu_S(x_0) = -e_2$. Moreover, up to rescaling $\xi$ we can suppose $\|\xi \|_\infty = 1$. Let
\[
B_1^+ = B_1 \cap \R^2_+\,.
\]
We then want to show that the set 
\begin{equation}\label{def:Nalpha}
N_\alpha := \{x\in B^+_1\,:\, |\xi(x)+e_2|\ge\alpha \}
\end{equation}
has density zero at $0$ for all $\alpha >0$. Argue by contradiction and suppose there exist $\alpha, \beta >0$ and a sequence of radii $\{r_k\}_k$ decreasing to zero, such that
\begin{equation}\label{eq:contr}
\frac{|N_\alpha \cap B_{r_k}|}{\pi r_k^2} \ge \beta\qquad \text{for all }k\,.
\end{equation}
Define $z_0(x):= \xi(x) +e_2$ and the sequence $z_k(y) = z_0(r_ky)$ for $k\in \N$. Since the second component of $z_k$ is $z_{k,2}(x) = \xi_2(r_k x)+1$ one easily sees that
\begin{equation}\label{eq:iiperzk}
z_{k,2}(x) \ge \frac{|z_k(x)|^2}{2}\,,
\end{equation}
for almost every $x\in \R^2$. 
By the definition of $N_\alpha$ and by~\eqref{eq:iiperzk}, the contradiction hypothesis~\eqref{eq:contr} reads equivalently as
\begin{equation}\label{eq:contrequiv}
\left|\left \{x\in r_k^{-1}B_1^+ \,:\, z_{k,2}(x) > \frac{\alpha^2}{2} \right\} \cap B_1\right| \ge \pi \beta\,.
\end{equation}
On top of that, $z_0 \in \DM$ with $\|z_0\|_\infty \le 2$. By Proposition~\ref{prop:zblow} the sequence $z_k$ defined above converges  in $L^\infty$-$w^*$ (up to subsequences, we do not relabel) to a vector field $\eta$ such that $\div(\eta) = 0$ on $\R^2_+$ and $[\eta \cdot (-e_2)] = [z_0 \cdot (-e_2)](0)$ on $\R^2_0$. We aim to show that $\eta$ satisfies the hypotheses~(i)--(iii) of Theorem~\ref{thm:rigidity2}. Were this the case, one would conclude $\eta \equiv 0$ in $\R^2_+$ and this would yield a contradiction with~\eqref{eq:contrequiv}. Indeed, taking $\chi_{B_1^+}$ as a test function we get
\begin{equation}\label{eq:contradiction}
\pi \frac{\alpha^2 \beta}{2} \le \int_{B_1^+}z_{k,2} \xrightarrow[k]{\phantom{6chara}} \int_{B_1^+} \eta_2\,.
\end{equation}
On the one hand, we know that hypothesis~(i) of Theorem~\ref{thm:rigidity2} is satisfied as $\div(\eta) = 0$ on $\R^2_+$.  On the other hand, as $[\eta \cdot \nu] = -[z_0 \cdot e_2](0)$ on $\R^2_0$, we get that
\[
[\eta \cdot \nu] = -[\xi \cdot e_2](0) - e_2 \cdot e_2 = \|\xi\|_\infty - 1 = 0\,,
\]
so that hypothesis~(ii) of Theorem~\ref{thm:rigidity2} holds as well. We are left to show that hypothesis~(iii) of Theorem~\ref{thm:rigidity2} is satisfied. Since $z_k$ is equibounded, by~\eqref{eq:iiperzk} we infer as well that $|z_k|^2$ converges (up to subsequences, we do not relabel) to some function $\zeta$  in $L^\infty$-$w^*$. Clearly, one has from~\eqref{eq:iiperzk} and the weak-$*$ convergence of $z_k$ and of $|z_k|^2$ that $\zeta \le 2\eta_2$ almost everywhere. We want to prove that the same holds with $|\eta|^2$ in place of $\zeta$ so to retrieve hypothesis~(iii) of Theorem~\ref{thm:rigidity2} with the choice $\phi(t) = t^2/2$.  Take a probability measure $f\, dx$ with $f\in L^1(\R^2)$. Then, by Jensen's inequality
\[
\int |z_k|^2 f\, dx \ge \left|\int z_k f\, dx\right|^2 = \left( \int z_{k,1} f\, dx\right)^2 +  \left( \int z_{k,2} f\, dx\right)^2\,.
\]
As $k\to \infty$, by the weak-$*$ convergence we get
\[
\int \zeta f\, dx \ge \left( \int \eta_1 f\, dx\right)^2 +  \left( \int \eta_2 f\, dx\right)^2\,.
\]
Thus, by letting $f\, dx$ toward the Dirac measure centered at $x$,~(iii) follows at once for almost-every point $x$ (more precisely, $x$ must be a Lebesgue point for the functions $\zeta, \eta_1, \eta_2$). A direct application of Theorem~\ref{thm:rigidity2} yields the desired contradiction.
\end{proof}

\begin{cor}\label{cor:maximality}
Let $\Omega\subset \R^2$ be weakly-regular and let $\xi\in \DM(\Omega)$. Then, for $\H^{1}$-a.e. $x_0\in \de \Omega$ such that $[\xi\cdot \nu_\Omega](x_0) = \|\xi\|_\infty$ one has
\[
\aplim_{x\to x_0^-} \xi(x) = \|\xi\|_\infty\, \nu_\Omega(x_0)\,.
\]
\end{cor}
\begin{proof}
Thanks to the Gauss--Green formula~\eqref{eq:GG} in the special case $\psi=1$, one deduces that the vector field $\tilde{\xi}$ defined as $\tilde{\xi}=\xi$ on $\Omega$ and $\tilde{\xi}=0$ on $\R^2\setminus \Omega$ belongs to $\DM(\R^2)$. The conclusion is achieved by applying Theorem~\ref{thm:main} with $\tilde{\xi}$ and $\de\Omega$ in place of, respectively, $\xi$ and $S$.
\end{proof}

\subsection{An application to capillarity in weakly-regular domains} 

The trace property that we have studied in the last section is motivated by the study of the boundary behaviour of solutions to the prescribed mean curvature equation in domains with non-smooth boundary (see~\cite{LS17}). Let us consider the vector field
\[
Tu = \frac{\nabla u}{\sqrt{1+|\nabla u|^2}}
\]
associated with any given $u\in W^{1,1}_{\text{loc}}(\Omega)$. We say that $u$ is a solution to the prescribed mean curvature equation if 
\begin{equation}\label{eq:PMC}\tag{PMC}
\div Tu = H\qquad \text{on }\Omega
\end{equation}
in the distributional sense, where $H$ is a prescribed function on $\Omega$. One of the main results of~\cite{LS17} is the following theorem.
\begin{thm*}[\cite{LS17}, Theorem 4.1]
	Let $\Omega\subset \R^n$ be a weakly-regular domain and let $H$ be a given Lipschitz function on $\Omega$. Assume that the necessary condition for existence of solutions to~\eqref{eq:PMC} holds, that is,
	\[
	\left|\int_A H(x)\, dx\right| < P(A),\qquad \text{for all $A\subset \Omega$ such that $0<|A| < |\Omega|$.}
	\] 
	Then, the following properties are equivalent.
	\begin{itemize}
		\item[{\bf (E)}] \textit{(Extremality)} $\left|\int_{\Omega}H\, dx\right| = P(\Omega)$.
		\item[{\bf (U)}] \textit{(Uniqueness)} \eqref{eq:PMC} admits a solution $u$ which is unique up to vertical translations.
		\item[{\bf (M)}] \textit{(Maximality)} $\Omega$ is maximal for~\eqref{eq:PMC}, i.e. no solution can exist in any domain strictly containing $\Omega$.
		\item[{\bf (V)}] \textit{(weak Verticality)} There exists a solution $u$ which is weakly-vertical at $\partial \Omega$, i.e. 
		\begin{equation*}
		[Tu \cdot \nu] = 1\qquad \text{$\H^{n-1}$-a.e. on }\de\Omega\,,
		\end{equation*}
		where $[Tu \cdot \nu]$ is the weak normal trace of $Tu$ on $\de\Omega$. 
	\end{itemize}
\end{thm*}
We remark that, in the relevant case of $H$ a positive constant, the extremality property~(E) is equivalent to $\Omega$ being a minimal Cheeger set (i.e.~$\Omega$ is the unique minimizer of the ratio $P(A) / |A|$ among all measurable $A\subset \Omega$ with positive volume, see for instance~\cite{Leo15, LNS17, LS17a, Par11, PS17}) and $H$ equals the Cheeger constant of $\Omega$. In dimension $n=2$, this \textit{extremal case} corresponds exactly to capillarity in zero gravity for a perfectly wetting fluid that partially fills a cylindrical container with cross-section $\Omega$. We also stress that the uniqueness property~(U) holds in this case without any prescribed boundary condition; this means that the capillary interface in $\Omega\times \R$ only depends upon the geometry of $\Omega$. Another important remark should be made on the verticality condition~(V), which corresponds to the tangential contact property that characterizes perfectly wetting fluids. In~\cite{LS17} this condition is obtained under the weak-regularity assumption on $\Omega$, which somehow justifies the presence of the weak normal trace in the statement (see for instance in Figure~\ref{fig:porosita} an example of non-Lipschitz, weakly-regular domain built in~\cite{LS17a} and covered by~\cite[Theorem 4.1]{LS17}).
\begin{figure}[t]
 \includegraphics[trim={1.5cm 5cm 1.5cm 4.5cm},clip, width=\textwidth]{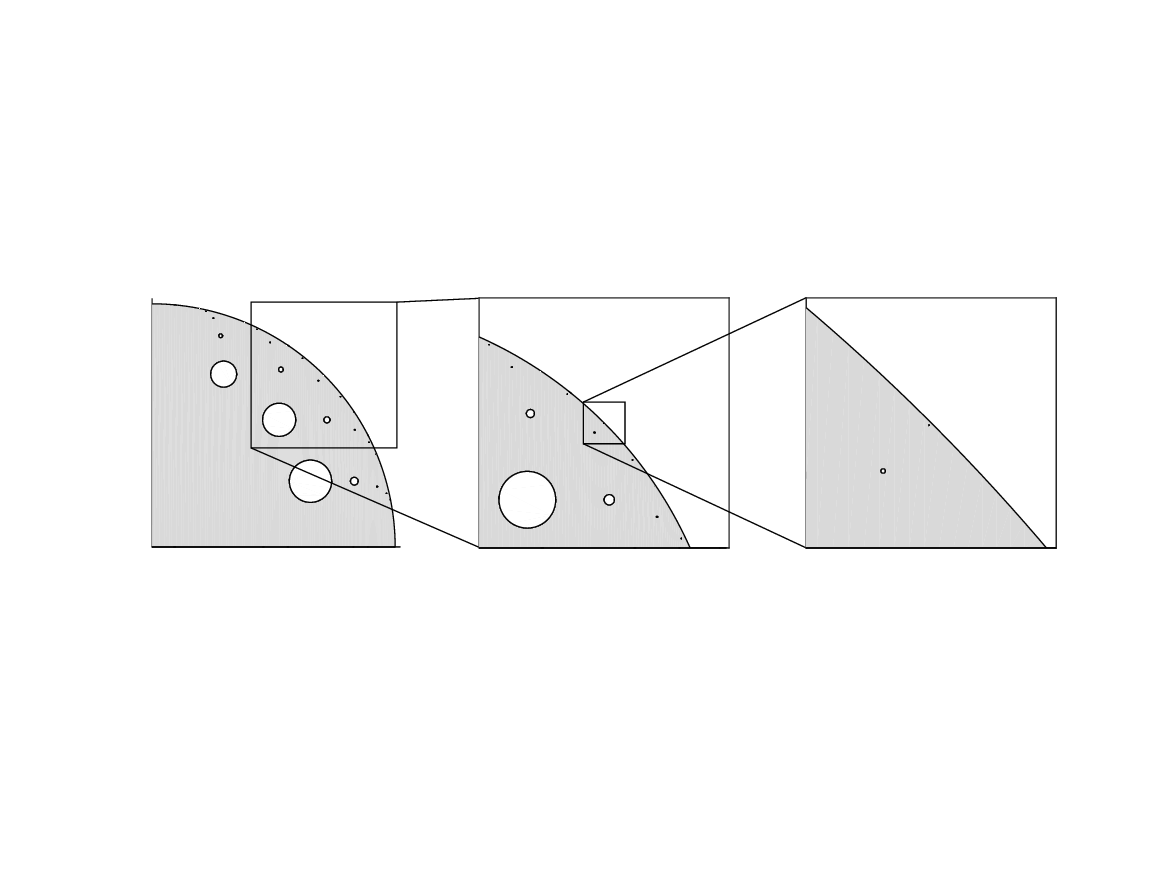}
  \caption{\label{fig:porosita} A non-Lipschitz, weakly-regular domain. Originally appeared in~\cite{LS17, LS17a}.}
  \end{figure} 
Nevertheless, in the physical three-dimensional case (i.e.~when $\Omega\subset \R^2$), the weak-verticality~(V) improves to strong-verticality, that is, the trace of $Tu$ exists and is equal to $\nu_\Omega$ almost-everywhere on $\de \Omega$ thanks to Theorem~\ref{thm:main} and to Corollary~\ref{cor:maximality}. However, we remark that the strategy of proof strongly relies on the rigidity property, which we have been able to prove only in dimension $n=2$. It is an open question whether the weak-verticality condition always improves to the strong-verticality given by the existence of the classical trace of $Tu$ at $\H^{n-1}$-almost-every point of $\de\Omega$, and more generally if Theorem~\ref{thm:main} holds in any dimension.

\bibliographystyle{plain}

\bibliography{bib_trace}

\end{document}